\documentclass[10pt,a4paper]{article}

\usepackage{amsfonts, amsmath, wasysym}
\usepackage{amssymb,amsthm,paralist}
\usepackage{latexsym}
\usepackage[usenames]{color}

\usepackage{url}

\definecolor{darkgreen}{rgb}{0,0.5,0}
\definecolor{darkred}{rgb}{0.7,0,0}
\usepackage[colorlinks, citecolor=darkgreen, linkcolor=darkred]{hyperref}

\theoremstyle{plain}
\newtheorem{lemma}{Lemma}[section]
\newtheorem{thm}[lemma]{Theorem}

\newtheorem{cor}[lemma]{Corollary}

\theoremstyle{definition}
\newtheorem{defn}[lemma]{Definition}

\newtheorem{rmk}[lemma]{Remark}

\numberwithin{equation}{section}

\newcommand{\m}{\ensuremath{{\cal M}}}

\newcommand{\cd}{\ensuremath{{\cal D}}}

\newcommand{\pl}[2]{{\frac{\partial #1}{\partial #2}}}

\newcommand{\pll}[2]{{\frac{\partial^2 #1}{\partial #2^2}}}

\newcommand{\de}{\delta}

\newcommand{\Om}{\Omega}

\newcommand{\ep}{\varepsilon}

\newcommand{\R}{\ensuremath{{\mathbb R}}}
\newcommand{\N}{\ensuremath{{\mathbb N}}}

\newcommand{\C}{\ensuremath{{\mathbb C}}}

\newcommand{\downto}{\downarrow}

\newcommand{\lap}{\Delta}

\newcommand{\grad}{\nabla}

\newcommand{\beq}{\begin{equation}}
\newcommand{\eeq}{\end{equation}}
\newcommand{\beqa}{\begin{equation}\begin{aligned}}
\newcommand{\eeqa}{\end{aligned}\end{equation}}
\newcommand{\brmk}{\begin{rmk}}
\newcommand{\ermk}{\end{rmk}}
\newcommand{\partref}[1]{\hbox{(\csname @roman\endcsname{\ref{#1}})}}

\newcommand{\Ric}{{\mathrm{Ric}}}

\newcommand*\dx{\ensuremath{\mathrm{d}x}}
\newcommand*\dy{\ensuremath{\mathrm{d}y}}
\newcommand*\dz{\ensuremath{\mathrm{d}z}}
\newcommand*\pddt{\ensuremath{\frac{\partial}{\partial t}}}
\newcommand*\ee{\ensuremath{\mathop{\mathrm{e}}\nolimits}}
\usepackage{soul}

\title{{\sc 
ricci flow of negatively curved incomplete surfaces
}}
\author{Gregor Giesen and Peter M. Topping}
\date{17 June 2009}

\begin{document}

\maketitle
\parskip=10pt

\begin{abstract}
We show uniqueness of Ricci flows starting at a surface
of uniformly negative curvature, with the assumption
that the flows become complete instantaneously.
Together with the more general existence result
proved in \cite{Top}, this settles the issue of
well-posedness in this class.
\end{abstract}

\section{Introduction}
\label{intro}

In 1982, Hamilton \cite{ham3D} introduced the study of 
Ricci flow, which evolves a Riemannian metric $g$ on a manifold 
\m\ under the nonlinear evolution equation
\beq
\label{RFeq}
\pl{}{t}g(t)=-2\,\Ric[g(t)].
\eeq
Hamilton proved that if \m\ is closed (i.e. compact and 
without boundary) then for any initial metric $g_0$,
there exist $T>0$ and a smooth Ricci flow $g(t)$ for 
$t\in [0,T]$, with $g(0)=g_0$. We also have uniqueness:
even if $T$ is reduced, there can be no other such flow.
(See also \cite{deturck}.)
Shi \cite{shi} and Chen-Zhu \cite{chenzhu} generalised
this to the case of noncompact \m\ in the case that 
the initial metric and all flows are assumed to be 
complete and with bounded curvature.

This theory left open the problem of starting a Ricci flow
in the more general situation that the initial metric is
incomplete. This possibility springs out when one 
contemplates, for example, restarting a Ricci flow after 
a finite-time singularity has occurred in the case that
\m\ has dimension at least $3$.

In \cite{Top}, the second author developed a very 
general existence theorem for Ricci flows which 
produces (as a special case) a Ricci flow starting
at any initial Riemannian surface of Gauss curvature
bounded above -- whether complete or not -- which 
distinguishes itself by being complete
at any strictly positive time. Evidence was given
in \cite{Top} to support the idea that this 
\emph{instantaneous completeness} should be the right
condition to guarantee uniqueness also.

In this paper, we show that this is the case under the
additional assumption that the upper bound for the Gauss
curvature is negative. We also demonstrate how the existence
issue is simpler in this case. 

\begin{thm}
  \label{thm:main}
Suppose $\m$ is any surface (i.e. a $2$-dimensional 
manifold without boundary)
equipped with a smooth Riemannian metric $g_0$
whose Gauss curvature satisfies $K[g_0]\leq -\eta<0$,
but which need not be complete.
Then there exists a {\bf unique} smooth Ricci flow $g(t)$ for
$t\in [0,\infty)$ with the following properties:
\begin{compactenum}[(i)]
\item
$g(0)=g_0$;
\item
$g(t)$ is complete for all $t>0$;
\item  
the curvature of $g(t)$ is bounded above for any compact
time interval within $[0,\infty)$;
\item
the curvature of $g(t)$ is bounded
below for any compact time interval within $(0,\infty)$.
\end{compactenum}
Moreover, this solution satisfies
$K[g(t)]\leq -\frac{\eta}{1+2\eta t}$ for $t\geq 0$
and
$-\frac{1}{2t}\leq K[g(t)]$ for $t>0$.
\end{thm}

Some discussion of what such flows look like can be found
in \cite{Top}. Generally, as $t\downto 0$, they blow
up in a manner reminiscent of reverse bubbling in the 
harmonic map heat flow (see \cite{revbub} and \cite{BDPVDH}).

The main difficulty in proving the new uniqueness part of
this result is that we do not assume directly any
control on the behaviour of any competing Ricci flow near 
spatial infinity. This control needs to be built up
by appealing to geometric results which can exploit
our completeness assumption, and combining the results
with a direct analysis of the conformal factor of the flows
more in the spirit of the
literature on the logarithmic fast-diffusion equation.
The main input from the previous literature comes from 
\cite{Top} and Yau's version of the Schwarz Lemma,
Theorem \ref{thm:yau} (see \cite{Yau73}).
We will also have to juggle two subtly different 
comparison principles: We prove a `geometric comparison
principle,' Theorem \ref{thm:geom-comp-principle} which 
compares two Ricci flows 
under the hypotheses that one of them is complete, and
certain curvature bounds are satisfied, and will also
repeatedly appeal to a standard `direct comparison principle,'
Theorem \ref{thm:direct-comp-principle} which compares 
certain Ricci flows without looking beyond the equation 
satisfied by their conformal factors, and in particular
without noticing their geometry.

\emph{Acknowledgements:} Both authors are partially supported
by The Leverhulme Trust.

\section{\emph{A priori} estimates on solutions} 
\label{section: apriori}

On a surface, the Ricci curvature of a metric $g$
takes the simple form $\Ric[g]=K[g]g$, where $K[g]$
represents the Gauss curvature. Therefore the Ricci flow
is the conformally invariant flow $\pl{g}{t}=-2K[g]g$
(which coincides with the Yamabe flow in this dimension).

If we choose a local complex coordinate $z=x+iy$ and
write the metric locally as $g=\ee^{2u}|\dz|^2$ 
(where $|\dz|^2=\dx^2+\dy^2$) then $K[g]=-\ee^{-2u}\lap u$
(where $\lap:=\pll{}{x}+\pll{}{y}$ is defined in terms
of the local coordinates) and we can write the Ricci flow as 
\begin{equation}
  \label{eq:ricci-flow-cf}
  \pl{u}{t}=\ee^{-2u}\lap u =-K[u],
\end{equation}
where we abuse notation here and in the sequel by
abbreviating $K[\ee^{2u}|\dz|^2]$ by $K[u]$.

The first observation to make about Theorem \ref{thm:main}
is that without loss of generality, we may assume that
$g_0$ is a conformal metric $\ee^{2u_0}|\dz|^2$ on the unit
disc $\cd\subset \C$. Indeed, we can lift $g_0$ to the universal
cover of $\m$, and since $g_0$ has uniformly negative 
Gauss curvature, the conformal type of this cover must
be $\cd$ (rather than $S^2$  or $\C$ which could be
ruled out using the Gauss-Bonnet Theorem or by applying
Corollary \ref{cor:yau2} below to large discs within $\C$,
respectively).
If we can establish both existence \emph{and} uniqueness
on the disc, then we can be sure to be able to quotient
the solution to give ultimately a unique solution on the original
surface.

Next we observe that by dilating $g_0$, and parabolically 
rescaling $g(t)$ (see \cite[\S 1.2.3]{RFnotes} for a discussion
of parabolic rescaling) we may assume that $\eta=1$ in 
Theorem \ref{thm:main}. 
These considerations motivate the following:

\begin{defn}
  \label{defn:admissible-soln}  
  A smooth Ricci flow
  $g(t)$ on $\cd$ for $t\in [0,T]$ 
  is called \textbf{admissible}, provided
  \begin{compactenum}[(i)]
  \item \label{cond:defn-admissible-soln-complete} 
    $g(t)$ is complete for $t>0$;
  \item \label{cond:defn-admissible-soln-upper-curv-bound}
    $K[g]\le C$ on $[0,T]\times\cd$;
  \item \label{cond:defn-admissible-soln-lower-curv-bound} 
    $K[g]\ge -C_\varepsilon$ on $[\varepsilon,T]\times\cd$ for all
    $\varepsilon\in(0,T]$.
  \end{compactenum}
\end{defn}

\begin{lemma}
  \label{lemma:apriori-estimates}
  Suppose $\ee^{2u_0}|\dz|^2$ is a smooth metric on the disc $\cd$ with 
  $K[u_0]\le -1$ and $\ee^{2w(t)}|\dz|^2$ is an
  admissible Ricci flow on $[0,T]\times\cd$ with initial condition $w(0)=u_0$.
  Then $w$ satisfies 
  \begin{compactenum}[(A)]
  \item $K[w] \ge -\frac1{2t}$,
  \item $w(t,x) \ge \ln\frac2{1-|x|^2} + \frac12\ln(2t)$,
  \end{compactenum}
  on $(0,T]\times\cd$, while on $[0,T]\times\cd$ we have
  \begin{compactenum}[(A)]
    \setcounter{enumi}{2}
  \item $w(t,x) \le \ln\frac2{1-|x|^2} +
    \frac12\ln\left(2t+1\right)$, 
  \item $w(t,x) \ge u_0(x) - Ct$.
  \end{compactenum}
\end{lemma}

The proof relies on the following special case of the 
Schwarz lemma of S.-T. Yau. For 
convenience we give a proof in Appendix \ref{chap:thm-yau}.

\begin{thm}[Schwarz-Pick-Ahlfors-\!Yau {\cite{Yau73}}]
  \label{thm:yau}
  Let $\left(\m_1,g_1\right)$ and $\left(\m_2,g_2\right)$ be two Riemannian
  surfaces without boundary. If
  \begin{compactenum}[(a)]
  \item $\left(\m_1,g_1\right)$ is complete,
  \item $K[g_1]\ge -a_1$ for some number $a_1\ge0$, and
  \item $K[g_2]\le -a_2<0$,
  \end{compactenum}
  then any conformal map $f:\m_1\to\m_2$ satisfies
  \[ f^*(g_2)\le\frac{a_1}{a_2} g_1. \]
\end{thm}
Setting $\m_1=\m_2$ to be a disc, $a_1=a_2=C>0$, 
$f=\mathrm{id}$ and either $g_1$ or $g_2$ to be $H$,
the complete metric of constant curvature $-C$, one obtains 
barriers for metrics of uniformly negative curvature.
The most significant consequence is the following lower bound.
\begin{cor}
  \label{cor:yau1}
  Let $g$ be a \emph{complete} conformal 
  Riemannian metric on a disc, whose Gauss
  curvature is bounded from below by a constant $-C<0$. If $H$ is 
  the complete conformal metric of constant curvature $-C$ on
  that disc, then 
  \[ H \le g. \]
\end{cor}
A further consequence (which also follows by a more elementary
comparison argument) is an upper bound for (possibly incomplete)
negatively curved surfaces:
\begin{cor}
  \label{cor:yau2}
  Let $g$ be a conformal Riemannian metric on a disc, 
  whose curvature is bounded from above by a constant $-C<0$. 
  If $H$ is the complete conformal metric of 
  constant curvature $-C$ on that disc, then 
  \[ g \le H. \]
\end{cor}

\begin{proof}[Proof of Lemma \ref{lemma:apriori-estimates}]
  The Gauss curvature obeys the equation
  \[ \frac{\partial K}{\partial t} = \Delta K + 2K^2,\]
  under Ricci flow (see for example \cite[Proposition 2.5.4]{RFnotes})
  and one may apply the comparison principle (for example
  \cite[Theorem 12.14]{Cho08}) if the flow is complete and its
  curvature is bounded.
  For $\varepsilon>0$ the conditions
  (\ref{cond:defn-admissible-soln-upper-curv-bound}) and
  (\ref{cond:defn-admissible-soln-lower-curv-bound}) of Definition 
  \ref{defn:admissible-soln} give such a
  uniform bound on $K[w(t)]$ restricted to the time interval
  $t\in[\varepsilon,T]$. Comparing to the solution of the ODE $\pddt k=2k^2$
  with the lower curvature bound $-C_\varepsilon$ as initial condition at time
  $t=\varepsilon$, yields 
  \[ K\left[w|_{[\varepsilon,T]}\right] \ge - \frac1{2(t-\varepsilon) +
    C_\varepsilon^{-1}} \ge -\frac1{2(t-\varepsilon)}. \]
  Letting $\varepsilon\to0$, one obtains (A).

  For any time $t>0$, the metric $\ee^{2w(t)}|\dz|^2$ is complete
  (\ref{cond:defn-admissible-soln-complete}) and has 
  curvature bounded from below (A). Using Corollary \ref{cor:yau1} one
  obtains directly (B), since the conformal metric of constant curvature
  $-\frac1{2t}$ on the disc is
  \[ 2t\left(\frac{2}{1-|x|^2}\right)^2|\dz|^2. \]

  For small $\delta>0$, consider $w|_{\overline{\cd_{1-\delta}}}$ and write the
  conformal factor of the Ricci flow on the disc of radius 
  $1-\delta$ with Gauss curvature initially $-1$ as
  \[ h_\delta(t,x) := \ln\frac{2(1-\delta)}{(1-\delta)^2-|x|^2} +\frac12\ln(2t+1). \]
  By Corollary \ref{cor:yau2} we have $w|_{\cd_{1-\delta}}(0,\cdot)\le
  h_\delta(0,\cdot)$. Furthermore $w|_{\overline{\cd_{1-\delta}}}$ and
  $h_\delta$ fulfil the requirements for the direct
  comparison principle (Theorem \ref{thm:direct-comp-principle}), thus 
  $w|_{\cd_{1-\delta}} \le h_\delta$ holds throughout $[0,T]\times\cd_{1-\delta}$. 
  Since $h_\delta$ is continuous in $\delta$, letting $\delta\to0$ yields (C).

  Finally, the upper curvature bound
  (\ref{cond:defn-admissible-soln-upper-curv-bound}) and the 
  evolution equation of $w$ gives
  \[ \pddt w(t,x) = -K[w(t,x)] \ge -C, \]
  which integrates to give (D).
\end{proof}

\section{Existence}

The existence of the solution given by Theorem \ref{thm:main}
is a special case of the more general existence theory
developed in \cite{Top}.
However, the proof can be streamlined in the case 
that $(\m,g_0)$ is conformally hyperbolic
(by which we mean that it can be made hyperbolic by a conformal change
of metric) and in this section we sketch this simplified
proof in the particular case that $(\m,g_0)$ has $K[g_0]\leq -1$
and is conformally the disc $\cd$. (In this paper we can 
reduce to this case by virtue of our uniqueness result 
as described in Section \ref{section: apriori}.)

\begin{thm}[{Existence, special case of \cite[Theorem 1.1]{Top}}]
  \label{thm:existence}
  Let $g_0$ be a smooth conformal metric on $\cd$
  (possibly incomplete) with $K[g_0]\leq -1$.
  Then there exists a smooth Ricci flow $G(t)$ on $\cd$,
  for $t\in[0,\infty)$ with $G(0)=g_0$, such that $G(t)$ is complete 
  for every $t>0$.
  The Gauss curvature of this \emph{instantaneously complete 
  solution} satisfies
  \[ -\frac1{2t} < K\bigl[G(t)\bigr] \le 
    -\frac1{2t+1}
    \qquad\text{for $t>0$.} \]
    Moreover, $G(t)$ is maximal in the sense that if
    $g(t)$ for $t\in[0,\ep]\subset [0,\infty)$ is another Ricci flow 
    with  $g(0)=g_0$, then 
    $$g(t) \leq G(t)$$
    for all $t\in[0,\ep]$.
\end{thm}

The properties described in Lemma \ref{lemma:apriori-estimates} also 
apply to these solutions.

\begin{proof}
We follow the basic strategy of \cite{Top}, constructing 
$G(t)$ as a limit of approximating Ricci flows on smaller
base manifolds. We make some simplification of the convergence,
and exploit what we know about the conformal type to simplify
the proof of instantaneous completeness in this special case.

Let $u_0:\cd\to\R$ be the conformal factor of $g_0$, that is,
$g_0=\ee^{2u_0}|\dz|^2$.

For each $k\in\N$, define $D_k:=\cd_{1-\frac1{k+1}}$
to be the disc of radius $1-\frac1{k+1}$,
and let $h_k:D_k\to\R$ defined by
\[
  h_k(x) :=
  \ln\frac{\frac{2}{k+1}}{\bigl(1-\frac{1}{k+1}\bigr)^2-|x|^2}\]
be the conformal factor of
the complete conformal metric of curvature $-k^2$ on $D_k$.
Note that $h_k$ is pointwise (weakly) decreasing 
in the sense that for all $x\in\cd$, and $k$ sufficiently
large so that $x\in D_k$, the sequence $h_k(x)$ is weakly
decreasing.

Loosely following \cite{Top}, given $\eta>0$ we 
choose a smooth cut-off 
function $\Psi:\R\to\R$ with the properties
that $\Psi(s)=0$ for $s\leq -\eta$, $\Psi(s)=s$ for $s\geq \eta$,
and $\Psi''(s)\geq 0$ for all $s$. Then 
$0\leq \Psi'\leq 1$ and $\Psi(s)\geq s$ for all $s$.
We use $\Psi$ to define the metric
$$\bar g_k = \ee^{2\Psi(h_k-u_0)}g_0$$
on $D_k$, 
which can be viewed as a smoothed-out `pointwise maximum' of
the metrics represented by $u_0$ and $h_k$.
Writing $\bar u_k:D_k\to\R$ for the conformal factor of 
$\bar g_k$, that is, $\bar g_k=\ee^{2\bar u_k}|\dz|^2$,
we have 
$$\bar u_k\geq u_0|_{D_k}\text{ and }\bar u_k\geq h_k.$$
Just as in \cite[\S 4]{Top}, we see that 
$K[\bar g_k]$ is bounded below (with lower bound dependent on $k$)
and abbreviating $w_k:= h_k-u_0|_{D_k}$, 
we compute the uniform upper curvature bound
\begin{align*}
  K[\bar g_k] &= -\ee^{-2\bigl(\Psi(w_k)+u_0\bigr)}
  \Delta\bigl(\Psi(w_k)+u_0\bigr) \\
  &= -\ee^{-2\bigl(\Psi(w_k)+u_0\bigr)}
  \left( \Psi''(w_k)\bigl|\nabla w_k\bigr|^2
    + \Psi'(w_k)\Delta(h_k-u_0) + \Delta u_0 \right)\\
  &\le -\ee^{-2\bigl(\Psi(w_k)+u_0\bigr)} \left(
    \Psi'(w_k)\Delta h_k + \bigl(1-\Psi'(w_k)\bigr)\Delta u_0
  \right)\\
  &= \ee^{-2\bigl(\Psi(w_k)+u_0\bigr)} \Bigl(
    \Psi'(w_k)\left(\ee^{2h_k}K[h_k]\right) 
    + \bigl(1-\Psi'(w_k)\bigr)\left(\ee^{2u_0}K[u_0]\right) 
  \Bigr)\\
  &= \Psi'(w_k)\ee^{-2\bigl(\Psi(w_k)-w_k\bigr)}
    K[h_k]
    + \bigl(1-\Psi'(w_k)\bigr)\ee^{-2\Psi(w_k)}
    K[u_0]
  \\
  &\le \ee^{-2\eta}\Bigl(\Psi'(w_k)K[h_k]
    + \bigl(1-\Psi'(w_k)\bigr)K[u_0]\Bigr)\\
  &\le -\ee^{-2\eta}.
\end{align*}
In the last-but-one line we used the facts that both
$K[h_k]$ and $K[u_0]$ are negative, and also that 
$\Psi(s)-s \le\eta$ where $\Psi'(s)\neq 0$ (i.e. for $s\geq -\eta$) 
and $\Psi(s)\le\eta$ where $\Psi'(s)\neq 1$ (i.e. for $s\leq \eta$).
The last line follows from the fact that both
$K[h_k]\leq -1$ and $K[u_0]\le -1$.

The conformal factors $\bar u_k$ are (weakly) decreasing 
(as the $h_k$ are decreasing) and
$$\lim_{k\to\infty} \bar u_k(x)=u_0(x).$$

Let $g_k(t)$ be the Ricci flow as given by Shi \cite{shi} with
$g_k(0)=\bar g_k$, on $D_k$ over a maximal time interval
$[0,T)$.
Since these Ricci flows are each complete with bounded
curvature, the maximum principle (as in Lemma 
\ref{lemma:apriori-estimates}) tells us that
$K[g_k(t)]\geq -\frac{1}{2t}$ for $t>0$ but also that 
$K[g_k(t)]\leq -\frac1{2t+\ee^{2\eta}}$ for $t\geq 0$.
In particular, we must have $T=\infty$ -- i.e. long-time
existence for each $g_k(t)$ --
since the curvature is known to blow up at a singularity
of a Ricci flow.

Let $u_k:[0,\infty)\times D_k\to\R$ be the conformal factor of $g_k$,
that is, $g_k=\ee^{2u_k}|\dz|^2$.
Since the conformal factors $u_k(0)=\bar u_k$ are decreasing
in $k$,
we can compare $u_k|_{\overline{D_{k-1}}}(t)$ and $u_{k-1}(t)$ using
Theorem \ref{thm:direct-comp-principle} to find that 
the sequence $u_k(t)$ is (weakly) decreasing in the
sense that at each point $x\in\cd$ and $t\geq 0$,
for sufficiently large $k$ so that $x\in D_k$
we have $u_k(t,x)$ (weakly) decreasing.

By virtue of \eqref{eq:ricci-flow-cf} we have
$$\pl{u_k(t)}{t}=-K[u_k(t)]\geq 0,$$
and hence $u_0(x)\leq \bar u_k(x) \leq u_k(t,x)$ for all
$x\in D_k$ and $t>0$, so it makes sense to define 
$u:[0,\infty)\times\cd\to\R$ by
$$u(t,x)=\lim_{k\to\infty}u_k(t,x),$$
and consider the corresponding metric flow $G(t):=\ee^{2u(t)}|\dz|^2$.

By parabolic regularity theory we can see that $G(t)$
will be a smooth Ricci flow inheriting the curvature estimates 
of $g_k(t)$ and satisfying $G(0)=g_0$.

To see the instantaneous completeness of $G(t)$, we compare
$u$ with the conformal factor 
$h(t,x):=\ln\frac2{1-|x|^2} + \frac12\ln(2t)$ of the 
`big-bang' Ricci flow which is the metric on $\cd$ of
constant curvature $-\frac{1}{2t}$ at time $t>0$.
Indeed, the comparison principle
of Theorem \ref{thm:direct-comp-principle}
tells us that
$h(t,x) \leq u_k(t,x)$
for all $x\in D_k$ and $t>0$, and 
therefore, by taking the limit $k\to\infty$, 
$h(t,x)\leq u(t,x)$
for all $x\in \cd$ and $t>0$.

The maximality of $G(t)$ follows from a similar comparison argument.
If $\tilde u:[0,\ep]\times\cd\to\R$ is the conformal factor of any
other Ricci flow with $\tilde u(0,\cdot)=u_0$, then 
the comparison principle 
of Theorem \ref{thm:direct-comp-principle}
tells us that
$\tilde u|_{D_k}(t,x) \leq u_k(t,x)$
for all $x\in D_k$ and $t\in [0,\ep]$, and 
therefore (taking $k\to\infty$)
$\tilde u(t,x)\leq u(t,x)$
for all $x\in \cd$ and $t\in [0,\ep]$.

We have almost finished, except that we appear to have constructed
a flow $G(t)$ for each $\eta>0$, and each of these is 
guaranteed only to have its Gauss curvature bounded above
by $-\frac1{2t+\ee^{2\eta}}$. However, it is not hard to see
that there can exist only one \emph{maximal} solution,
and so all of the flows $G(t)$ must be identical.
At this point we may take the limit $\eta\downto 0$
and deduce that
$K[G(t)]\leq -\frac1{2t+1}$ for $t\geq 0$, which completes
the proof.
\end{proof}

\section{Uniqueness}

The following theorem states the uniqueness part of the main Theorem
\ref{thm:main} and concludes its proof.
\begin{thm}
  \label{thm:uniqueness}
  Let $\ee^{2u_0}|\dz|^2$ be a smooth metric on the unit
  disk $\cd$ satisfying the upper curvature bound $K[u_0]\le-1$. 
  For some $T>0$ let $\ee^{2v(t)}|\dz|^2$ be an admissible Ricci flow
  (Definition \ref{defn:admissible-soln}) with
  $v(0)=u_0$.  
  Then $\ee^{2v(t)}|\dz|^2$ is unique among such instantaneously complete
  solutions. 
\end{thm}
The proof relies on the following \emph{geometric} comparison result.
\begin{thm}[Geometric comparison principle]
  \label{thm:geom-comp-principle}
  Suppose $\bigl(\m, g_1(t)\bigr)$ and $\bigl(\m, g_2(t)\bigr)$ 
  are two conformally equivalent Ricci flows on some time interval
  $[0,T]$, and define $Q:[0,T]\times\m\to\R$ to be 
  the function for which
  $g_1(t)=\ee^{2Q(t)}g_2(t)$. 
  Suppose further that $g_2(t)$ is complete for each $t\in [0,T]$
  and that for some constant $C\ge0$ we have
  \[ \text{(i) } \bigl|K[g_2]\bigr|\le C,\qquad 
  \text{(ii) } K[g_1]\le C, \qquad
  \text{(iii) } Q\le C \]
  on $[0,T]\times\cd$. If $g_1(0)\le g_2(0)$, then 
  $g_1(t)\le g_2(t)$ for all $t\in [0,T]$.
\end{thm}
\begin{proof}
With respect to a local complex coordinate $z$, let us write
$g_1(t)=\ee^{2u(t)}|\dz|^2$ and $g_2(t)=\ee^{2\tilde v(t)}|\dz|^2$
for some locally defined functions $u(t)$ and $\tilde v(t)$,
and note then that $Q=u-\tilde v$.
  Since $g_1(t)$ and $g_2(t)$ are Ricci flows, we get 
  \[ \frac{\partial(u-\tilde v)}{\partial t} = \ee^{-2u}\Delta u - \ee^{-2\tilde
    v}\Delta\tilde v  
  = \left(\ee^{-2u}-\ee^{-2\tilde v}\right)\Delta u + \ee^{-2\tilde
    v}\Delta(u-\tilde v). \] 
  Writing $\Delta_{g_2(t)}$ for the Laplace-Beltrami operator 
  with respect
  to the metric $g_2(t)$, we obtain, 
  where $Q>0$ (i.e. where $u>\tilde v$)
  \begin{align*}
  \left(\pddt-\Delta_{g_2(t)}\right)Q
    &=
    \left(\pddt-\ee^{-2\tilde v}\Delta\right)(u-\tilde v)\\
    &= \left(\ee^{-2u}-\ee^{-2\tilde v}\right)\Delta u 
    = (u-\tilde v)\frac{\ee^{-2u}-\ee^{-2\tilde v}}{u-\tilde v}\Delta u \\
    &= (u-\tilde v)(-2)\ee^{-2\xi}\Delta u 
    = 2(u-\tilde v)\ee^{2(u-\xi)}\left(-\ee^{-2u}\Delta u\right) \\
    &= 2\ee^{2(u-\xi)}K[g_1]\,(u-\tilde v)\\
    & \le 2\ee^{2C}C(u-\tilde v)=(2C\ee^{2C}) Q,
  \end{align*}
  where at each point, $\xi$ was chosen between $u$ and $\tilde v$ according to
  the mean value theorem. Applying the weak maximum principle 
  to $Q$
  (see for example \cite[Theorem 12.10]{Cho08}
  with $g_2(0)$ as complete
  background metric with bounded curvature and $g_2(t)$ as
  one-parameter family of complete metrics) keeping in
  mind that $Q(0,\cdot)\leq 0$, we conclude 
  that $Q\leq 0$ throughout $[0,T]\times \m$ as desired.
\end{proof}

\begin{proof}[Proof of Theorem \ref{thm:uniqueness}]
  Theorem \ref{thm:existence} provides the existence of such an admissible
  solution $\ee^{2u(t)}|\dz|^2$ with $u(0)=u_0$. Let $\ee^{2v(t)}|\dz|^2$ be
  any another admissible solution with the same initial condition $v(0)=u_0$.
  From Theorem \ref{thm:existence} we know that $u(t)$
  is maximal among such instantaneously complete solutions, 
  and in particular, $v(t)\le
  u(t)$ for all $t\in[0,T]$. Hence it remains to show the converse
  inequality $u(t)\le v(t)$.

  Let $C>0$ be the uniform upper bound of the curvature of
  $v$: $K[v(t)]\le C$ for all $t\in[0,T]$, which exists
  since $v$ is admissible. For small $\delta\in(0,T)$ define
  \[ \tilde{v}(t,x) := v\left(\ee^{-2C\delta}(t+\delta),x\right) 
  + C\delta
  \qquad\quad \text{for $(t,x)\in[0,T-\delta]\times\cd$,} \]
  which is a slight adjustment of $v$, 
  again a solution to the Ricci flow:
  \[ \left(\pddt\tilde v - \ee^{-2\tilde v}\Delta\tilde v\right)(t,x)
  = \ee^{-2C\delta}\left(\pddt v-\ee^{-2v}\Delta v\right)
  \left(\ee^{-2C\delta}(t+\delta),x\right)=0.  \]
Our aim is to show that $u$ is a lower bound for $\tilde v$, and
hence (by taking $\de\downto 0$) also a lower bound for $v$ 
as desired.
To do this, we wish to apply Theorem \ref{thm:geom-comp-principle}
to the Ricci flows $g_1(t)$ and $g_2(t)$ generated by the 
conformal factors $u$ and $\tilde v$ respectively. 

First, note that $g_2(t)$ is complete for all $t\in [0,T-\de]$
since $\ee^{2 v}|\dz|^2$ is an admissible Ricci flow
and is therefore complete for all $t\in (0,T]$.
Furthermore, $g_2(t)$ has upper and lower curvature bounds:
  \begin{equation}
    \label{eq:uniq-Ktv}
    \bigl|K[\tilde v]\bigr| 
    \le 
    \sup_{\left[\ee^{-2C\delta}\delta,T\right]\times\cd}
    \ee^{-2C\delta}\bigl|K[v]\bigr| < \infty, 
  \end{equation}
so hypothesis (i) of Theorem \ref{thm:geom-comp-principle} is satisfied.
The upper bound for the curvature of $g_1(t)=\ee^{2u}|\dz|^2$ 
required by hypothesis (ii) of Theorem \ref{thm:geom-comp-principle}
follows since $g_1(t)$ was constructed to be admissible.

Next we verify hypothesis (iii) of Theorem \ref{thm:geom-comp-principle},
namely that $u-\tilde v$ is bounded from above.
Applying (C) of Lemma \ref{lemma:apriori-estimates} to 
$\ee^{2u}|\dz|^2$, we find that 
\begin{equation}
\label{eq:u}
u(t,x)\leq \ln\frac2{1-|x|^2} + \frac12\ln\left(2t+1\right),
\end{equation}
for $t\in [0,T]$, 
while (B) of Lemma \ref{lemma:apriori-estimates} applied to
$\ee^{2v}|\dz|^2$ gives
$$v(t,x)\geq \ln\frac2{1-|x|^2} + \frac12\ln\left(2t\right),$$
for $t\in (0,T]$ and hence that
\begin{align}
\tilde v(t,x) &\geq \ln\frac2{1-|x|^2} + 
\frac12\ln\left(2\ee^{-2C\de}(t+\de)\right)+C\de \notag\\
&= \label{eq:tildev}
\ln\frac2{1-|x|^2} + 
\frac12\ln\bigl(2(t+\de)\bigr)
\end{align}
for $t\in [0,T-\de]$. Subtracting \eqref{eq:tildev} from 
\eqref{eq:u}, we find that
$$u-\tilde v \leq \frac12\ln\left(2T+1\right)
-\frac12\ln\left(2\de\right)$$
as desired.

The final hypothesis of Theorem \ref{thm:geom-comp-principle}
to verify is that $g_1(0)\leq g_2(0)$, i.e. that 
$u(0,\cdot)\leq \tilde v(0,\cdot)$.
But 
$$\tilde v(0,\cdot)=v\left(\ee^{-2C\de}\de,\cdot\right)+C\de
\geq u_0 -C\ee^{-2C\de}\de+C\de\geq u_0=u(0,\cdot)$$
by part (D) of Lemma \ref{lemma:apriori-estimates}
as desired.

We may therefore apply Theorem \ref{thm:geom-comp-principle}
over the time interval $[0,T-\de]$ to deduce that
$u(t)\leq \tilde v(t)$ for all $t\in [0,T-\de]$.
Hence, given any $(t,x)\in [0,T)\times \cd$,
we conclude
\[ u(t,x)\leq \lim_{\de\downto 0}\tilde v(t,x)= v(t,x). \qedhere \]
\end{proof}

\begin{appendix}
\section{Comparison principle}

In this appendix we clarify the statement and proof of one
of the many variants of the standard weak maximum principle.

\begin{thm}[Direct comparison principle]
  \label{thm:direct-comp-principle}
  
  Let $\Omega\subset\mathbb{R}^2$ be an open, bounded domain and for some
  $T>0$ let $u\in C^{1,2}\bigl((0,T)\times\Omega\bigr)\cap
  C\bigl([0,T]\times\bar\Omega\bigr)$ and $v\in
  C^{1,2}\bigl((0,T)\times\Omega\bigr)\cap
  C\bigl([0,T]\times\Omega\bigr)$ both be solutions of the Ricci flow equation \eqref{eq:ricci-flow-cf}
  for the conformal factor of the metric.
  Furthermore, suppose that  for each $t\in[0,T]$ we have
  $v(t,x)\to\infty$ as $x\to\partial\Omega$. 
  If $v(0,x)\ge u(0,x)$ for all $x\in\Omega$, then $v\ge u$ on
  $[0,T]\times\Omega$.
\end{thm}
\begin{proof}
  For every $\varepsilon>0$ consider
  \[ v_\varepsilon(t,x) := v\left(\frac1\varepsilon\ln(\varepsilon t+1),x\right)
  +\frac12\ln(\varepsilon t+1)\qquad\text{for all }
  (t,x)\in[0,T]\times\Omega, \]
  which is well-defined since $\frac1\varepsilon\ln(\varepsilon t+1)\le
  t$ for all $t\ge0$. Observe that $v_\varepsilon$ is a slight modification of
  $v$, with $v_\varepsilon(0,\cdot)=v(0,\cdot)$, and $v_\varepsilon$ converges pointwise
  to $v$ as $\varepsilon\to0$, but in contrast to $v$ it is a
  strict supersolution of
  the Ricci flow \eqref{eq:ricci-flow-cf}:
  \begin{align}
    \label{eq:rf-cf-mod}
    \left(\pddt v_\varepsilon-\ee^{-2v_\varepsilon}\Delta
      v_\varepsilon\right)(t,x) 
    &= \frac1{\varepsilon t+1}\left(\pddt v-\ee^{-2v}\Delta
      v\right)\left(\frac1\varepsilon\ln(\varepsilon t+1),x\right)
    + \frac\varepsilon{2(\varepsilon t+1)} \notag\\
    &= \frac\varepsilon{2(\varepsilon t+1)} >0 
    \qquad\qquad \text{for }
    (t,x)\in[0,T]\times\Omega.
  \end{align}
  We are going to prove $(v_\varepsilon-u)\ge0$ on
  $[0,T]\times\Om$ and conclude the theorem's statement by letting
  $\varepsilon\to0$. 
  Since by hypothesis $u$ is continuous on $[0,T]\times\bar\Omega$ and
  $v_\varepsilon(t,\cdot)$ blows up near the boundary $\partial \Om$
  for each $t\in[0,T]$, we have  
  \[ (v_\varepsilon-u)(t,x)\to\infty\quad\text{as}\quad x\to\partial\Omega, \]
  for every time $t\in[0,T]$ and hence $(v_\varepsilon-u)(t,\cdot)$
  attains its infimum in $\Omega$.
  Now assume that $(v_\varepsilon-u)$ becomes negative in $[0,T]\times\Omega$,
  and define the time $t_0$ at which $(v_\varepsilon-u)$
  first becomes negative by 
  \[ t_0 := \inf\Bigl\{t\in[0,T]:
  \min_{x\in\Omega}(v_\varepsilon-u)(t,x)<0 \Bigr\}\in[0,T). \]
  Picking any minimum $x_0\in\Omega$ of $(v_\ep-u)(t_0,\cdot)$,
  we have  
  \[ (v_\varepsilon-u)(t_0,x_0) = 0,\quad
  \Delta(v_\varepsilon-u)(t_0,x_0)\ge0\quad\text{ and }\quad
  \pddt(v_\varepsilon-u)(t_0,x_0)\le0.\]
  Subtracting the Ricci flow equation
  \eqref{eq:ricci-flow-cf} from \eqref{eq:rf-cf-mod} 
  at this point $(t_0,x_0)$, we find
  \begin{align*}
    0 &< 
    \left(\pddt v_\varepsilon - \ee^{-2v_\varepsilon}\Delta
      v_\varepsilon\right)(t_0,x_0)    
    -\left(\pddt u - \ee^{-2u}\Delta u\right)(t_0,x_0) \\
    &= \pddt(v_\varepsilon -u)(t_0,x_0) -
    \ee^{-2u(t_0,x_0)}\Delta(v_\varepsilon-u)(t_0,x_0) \le 0,
  \end{align*}
  which is a contradiction.
  Therefore $v_\varepsilon\ge u$ on $[0,T]\times\Omega$, and
  the corresponding statement for $v$ follows by letting $\varepsilon\to0$.
\end{proof}

\section{Yau's Schwarz Lemma}
\label{chap:thm-yau}
For convenience, we prove now the Schwarz lemma of Yau
(Theorem \ref{thm:yau}).
The proof uses the following generalised maximum principle by 
Omori, whose
proof was simplified by Yau in \cite[Theorem 1, p. 206]{Yau75}:
\begin{thm}\cite[Theorem A$'$, p. 211]{Omo66}
  \label{thm:omori}
  On a complete Riemannian surface $(\m,g)$ with Gaussian curvature bounded
  from below, let $f$ be a $C^2$-function 
  which is bounded above.
  Then, for an arbitrarily point $p\in\m$ and for any $\varepsilon>0$,
  there exists a point $q$ depending on $p$ such that
  \[ \text{(i) } \Delta_gf(q)<\varepsilon;\qquad
  \text{(ii) } |\grad f(q)|_{g}<\varepsilon;\qquad
  \text{(iii) } f(q)\ge f(p).\]
\end{thm}

The essential idea \cite{Yau75} to find $q$ is to imagine the point 
$\bigl(q,f(q)\bigr)$ on the graph of $f$ in $\m\times\R$ which
is closest to the point $(p,k)$ for some enormous $k\gg 1$.

\begin{proof}[Proof of Theorem \ref{thm:yau}]
  By dilating $g_1$ and $g_2$, we may assume $a_1=1=a_2$, that is
  $K[g_1]\ge-1 \ge K[g_2]$. 
Since we only need the theorem in the case that $f$ is 
strictly conformal, we will assume this in the proof and leave
the minor adjustments required for the full theorem to
the reader.\footnote{In the weakly conformal case
    $f$ might either be constant (nothing to prove) or have isolated singular
    points $P:=\{p_1, p_2,\ldots\}$. The function $w$ we define
    in \eqref{eq:def-w} will then have logarithmic
    singularities on $P$, but will be strictly negative 
    close to such singularities and the $\tilde{w}$ 
    of the proof could be
    adjusted to a smooth function on the whole of $\m_1$
    (including $P$) without altering anything where $w$ is
    positive.}
  Define $w\in C^\infty(\m_1)$ by 
  \begin{equation}
    \label{eq:def-w}
    f^*(g_2)=\ee^{2w}g_1.
  \end{equation}
  It remains to show that $w\le0$. Assume instead that 
  there exists $p\in\m_1$ with
  $w(p)>0$. Then we can choose an $\varepsilon\in(0,1)$ such
  that 
  \begin{equation}
    \label{eq:choice-eps}
    \varepsilon < \frac{\ee^{w(p)}-\ee^{-w(p)}}{1+\ee^{w(p)}}. 
  \end{equation}
  Now define 
  $\tilde w(x):= -\ee^{-w(x)}$ 
  for all $x\in\m_1$. 
  Since $(\m_1,g_1)$ is complete with curvature bounded from below and 
  $\tilde w$ is bounded above, we may apply Theorem \ref{thm:omori}
  to find a point $q\in\m_1$ with
  \begin{equation}
    \label{eq:gsl-Omega}
    \Delta_{g_1}\tilde w(q) < \varepsilon,\qquad
    |\nabla\tilde w|_{g_1}^2(q)<\varepsilon\quad\text{ and }\quad
    \tilde w(q) \ge \tilde w(p).
  \end{equation}
  Since 
  $x\mapsto -\ee^{-x}$ 
  is strictly increasing, we also have
  $w(q)\ge w(p)>0$. Now compute 
  \begin{align}
    \label{eq:yau-1}
    \Delta_{g_1}\tilde w
    &=\ee^{-w}\Delta_{g_1} w - \ee^{-w}|\nabla w|^2_{g_1} \notag\\
    &= \ee^{-w}\Delta_{g_1} w - \ee^{w}|\nabla \tilde w|^2_{g_1}.
  \end{align}
By computing with respect to a local complex coordinate,
we find that 
$$\lap_{g_1}w=-\ee^{2w} K[g_2]\circ f+K[g_1],$$
which together with the curvature estimates $-K[g_2]\geq 1$ and 
$K[g_1]\geq -1$ gives 
$\ee^{-w}\lap_{g_1}w\geq \ee^{w}-\ee^{-w}$, and so \eqref{eq:yau-1}
improves to
$$\Delta_{g_1}\tilde w \geq 
\ee^{w}(1-|\nabla \tilde w|^2_{g_1})
-\ee^{-w}.$$
Evaluating at $q$, using \eqref{eq:gsl-Omega} and the
fact that $w(q)\ge w(p)$, we obtain
$$\ep>\Delta_{g_1}\tilde w (q)\geq 
\ee^{w(q)}(1-\ep)-\ee^{-w(q)}
\geq \ee^{w(p)}(1-\ep)-\ee^{-w(p)}$$
and hence 
$$\ep>\frac{\ee^{w(p)}-\ee^{-w(p)}}{1+\ee^{w(p)}}$$
which contradicts \eqref{eq:choice-eps}.
\end{proof}

\end{appendix}

{\sc mathematics institute, university of warwick, coventry, CV4 7AL,
uk}\\
Giesen: \href{mailto:g.giesen@warwick.ac.uk}{g.giesen@warwick.ac.uk}\\
Topping: \url{http://www.warwick.ac.uk/~maseq}
\end{document}